\documentclass[oneside]{amsart}

\usepackage{color}
\usepackage{soul}

\makeatletter
\providecommand{\tabularnewline}{\\}

\theoremstyle{plain}
\newtheorem{thm}{Theorem}[section]
\newtheorem*{thm*}{Theorem}
\newtheorem{lem}[thm]{Lemma}
\newtheorem{prop}[thm]{Proposition}
\newtheorem{defi}[thm]{Definition}

\theoremstyle{remark}
\newtheorem{remark}[thm]{Remark}

\newtheorem*{acknowledgement*}{Acknowledgement}

\newcommand{\Dtriangle}[3]{%
	\setlength{\unitlength}{.08cm}
	\rule[-3\unitlength]{0pt}{18\unitlength}
	\begin{picture}(18,7)(0,3)
		\put(4,4){\circle{2}}
		\put(5,4){\line(1,0){8}}
		\put(14,4){\circle{2}}
		\put(4.4472,4.8944){\line(1,2){4.1056}}
		\put(9,14){\circle{2}}
		\put(13.5528,4.8944){\line(-1,2){4.1056}}
		\put(2,3){\makebox[0pt][r]{\scriptsize #1}}
		\put(9,17){\makebox[0pt]{\scriptsize #1}}
		\put(16,3){\makebox[0pt][l]{\scriptsize #1}}
		\put(6,9){\makebox[0pt][r]{\scriptsize #2}}
		\put(12.5,9){\makebox[0pt][l]{\scriptsize #3}}
		\put(9,1){\makebox[0pt]{\scriptsize #2}}
	\end{picture}
}

\newcommand{\corr}[2]{#2}
\newcommand{\FF}{\mathbb{F}}
\newcommand{\clabu}{class of type B}
\newcommand{\bu}{of type B}
\newcommand{\toba}{\mathfrak{B}}
\newcommand{\etalchar}[1]{$^{#1}$}
\newcommand{\tr}{\operatorname{tr}}

\makeatother

\begin{document}

\title{On Nichols Algebras over $\mathbf{PGL}(2,q)$ and $\mathbf{PSL}(2,q)$}

\author[Freyre, Gra\~na, Vendramin]{Sebasti\'an Freyre, Mat\'{\i}as Gra\~na, Leandro Vendramin}

\address{SF, MG, LV: Departamento de Matem\'atica - FCEyN \\
Universidad de Buenos Aires \\
Pab I - Ciudad Universitaria \\
(1428) Buenos Aires - Argentina}


\address{LV: Instituto de Ciencias\\
Universidad Nacional General Sarmiento\\
J. M. Gutierrez 1150 - Los Polvorines (1613) \\
Provincia de Buenos Aires - Argentina}

\address{}
\email{(sfreyre|matiasg|lvendramin)@dm.uba.ar}

\thanks{This work was partially supported by CONICET and ANPCyT (Argentina)}

\begin{abstract}
We compute necessary conditions on Yetter-Drinfeld modules over the groups
$\mathbf{PGL}(2,q)=\mathbf{PGL}(2,\FF_q)$ and
$\mathbf{PSL}(2,q)=\mathbf{PSL}(2,\FF_q)$ to generate finite dimensional
Nichols algebras.  This is a first step towards a classification of pointed
Hopf algebras with group of group-likes isomorphic to one of these groups. 

As a by-product of the techniques developed in this work, we prove that there
is no non-trivial finite-dimensional pointed Hopf algebra over the Mathieu
groups $M_{20}$ and $M_{21}=\mathbf{PSL}(3,4)$.
\end{abstract}

\maketitle
\section{Introduction}

The classification program for finite dimensional pointed Hopf algebras
comprises two different cases, according to whether the grouplikes form an
abelian or a nonabelian group. These two worlds have recently begun to approach
each other (\cite{ARXIV:0702559,ARXIV:0608701,ARXIV:0511020,ARXIV:0703498,
fantino-2007}).  Explicitly, the classification obtained by Heckenberger of
finite dimensional Nichols algebras over abelian groups
\cite{ARXIV:0509145,ARXIV:0605795,MR2207786} turned out to be a powerful tool
for the nonabelian case also.

Let us recall that, according to the lifting procedure \cite{AS}, to classify
pointed Hopf algebras with a specific group of grouplikes, the key step is to
compute the Nichols algebras generated by Yetter-Drinfeld modules over the
group.  With the abelian tools at hand, it is possible to rule out, for a given
group (or a family of groups), a large class of Yetter-Drinfeld modules which
can be shown to produce infinite dimensional Nichols algebras. Therefore, the
next step to classify pointed Hopf algebras over those groups is to study the
Nichols algebras produced by the remaining modules.

One of the key tools to produce Yetter-Drinfeld modules over groups is that of
racks and $2$-cocycles. They allow to work with the braided vector spaces
without having to resort to the groups. However, when trying to classify pointed
Hopf algebras parametrized by groups, sometimes the rack $2$-cocycles are ``too
general''. When the group is fixed, a conjugacy class can not hold any
$2$-cocycle, but only some of them. Many times, one can prove that all Nichols
algebras produced by a conjugacy class within a group are infinite-dimensional,
while one is not able to prove that, as a rack, the same thing happens with any
$2$-cocycle. Therefore, we develop the concept of \emph{\clabu}.
We use the same notation as in \cite{ARXIV:0703498}. In particular, $C_G(g)$ is 
the centralizer of $g\in G$ and $\toba(\mathcal{C},\rho)$ stands for the Nichols
algebra produced by the Yetter-Drinfeld module $M(\mathcal{C},\rho)$ (see
below). 

\begin{defi}
	Let $G$ be a group, $g\in G$ and $\mathcal{C}=\{xgx^{-1}:x\in G\}$ its
	conjugacy class. We say that $\mathcal{C}$ is a \clabu\ if for any
	representation $\rho\in\widehat{C_G(g)}$ the Nichols algebra
	$\toba(\mathcal{C},\rho)$ is infinite dimensional.
	We also say that $g$ is \bu\ if its conjugacy class is.
	We say that the group $G$ is \bu\ if all its conjugacy classes are.
\end{defi}

\begin{remark}
	By the lifting procedure, if $G$ is \bu\ then any finite dimensional pointed
	Hopf algebra with group of grouplikes (isomorphic to) $G$ is (isomorphic to)
	the group algebra of $G$.
\end{remark}

\begin{remark}\label{rem:tbym}
	It is proved in \cite{afgv} that if $g\in G$ is \bu\ and $f:G\hookrightarrow H$
	is a monomorphism of groups then $f(g)$ is \bu\ in $H$. This is the reason
	why this concept turns out to be a powerful tool.
\end{remark}

In \cite{ARXIV:0703498} Nichols algebras over $\mathbf{GL}(2,q)$ and
$\mathbf{SL}(2,q)$ were studied. In that paper it is proved that
$\mathbf{SL}(2,2^n)$ is \bu\ for $n>1$.  In this paper we deal with the groups
$\mathbf{PGL}(2,q)$ and $\mathbf{PSL}(2,q)$ for $q$ a power of an odd prime
number (recall that if $q$ is even then
$\mathbf{PSL}(2,q)=\mathbf{PGL}(2,q)=\mathbf{SL}(2,q)$).  For definitions and
elementary properties of these groups, see for example \cite{MR1369573}.

One of the main results of this work is Theorem~\ref{th:mibu}.

\begin{thm}\label{th:mibu}
	The Mathieu groups $M_{20}$ and $M_{21}=\mathbf{PSL}(3,4)$ are \bu.
\end{thm}
It is interesting to note that $M_{20}$ is non-simple. Also, notice that this is
the first example of a non-simple group which is \bu.  Other results of this
work are (see notations below):

\begin{thm}
	Let $G=\mathbf{PGL}(2,q)$ (see Table~\ref{tab:pgl} in \S\ref{se:pgl2q} for
	the conjugacy classes of $G$). Then, class $\mathcal{C}_2$ is \bu. Also, the
	conditions to obtain finite-dimensional Nichols algebras on the
	representations for classes $\mathcal{C}_3$, $\mathcal{C}_4$,
	$\mathcal{C}_5$ and $\mathcal{C}_6$ are given in propositions
	\ref{pro:pgl_c3}, \ref{pro:pgl_c4}, \ref{pro:pgl_c5} and \ref{pro:pgl_c6}. 
\end{thm}

\begin{thm}
	Let $G=\mathbf{PSL}(2,q)$ (see Tables~\ref{tab:psl_q1} and \ref{tab:psl_q3}
	in \S\ref{se:psl2q}). Then, classes $\mathcal{C}_i$ are \bu\ for $i<6$.
	Also, the conditions to obtain finite-dimensional Nichols algebras on the
	representation for class $\mathcal{C}_6$ are given in
	Proposition~\ref{pro:psl_c6_q1} if $q\equiv1\pmod4$, or in
	Proposition~\ref{pro:psl_c6_q3} if $q\equiv3\pmod4$.
\end{thm}

\section{Preliminaries}

\subsection{Notations}
As said, we use the same notation as in \cite{ARXIV:0703498}.  The number $p$
will be an odd prime number and $q$ a power of $p$, $E=\mathbb{F}_{q^{2}}$ will
be the quadratic extension of $\mathbb{F}_{q}$, $\overline{x}=x^q$ will be the
Galois conjugate of $x\in E$, $k$ will be an algebraically closed field of
characteristic zero, and we will write $\mathcal{R}_{n}$ for the set of
primitive $n$-th roots of $1$ in $k$. If $q\equiv1\pmod4$ we write $\pm i$ for
the square roots of $-1$ in $\mathbb{F}_q$.
Following the notation in \cite{jeffreyadams}, we fix an element
$\Delta\in\FF_q\setminus\FF_q^2$ and an element $\delta\in E$ which is a square
root of $\Delta$. We consider elements $z\in E$ being written as $z=x+\delta y$,
with $x,y\in\FF_q$.

If $G$ is a finite group, $\mathcal{C}$ is a conjugacy class of $g\in G$, and
$\rho$ is an irreducible representation of the centralizer $C_{G}(g)$, we write
$\mathfrak{B}(\mathcal{C},\rho)$ (or $\mathfrak{B}(\mathcal{C})$ if no confusion
can arise) for the Nichols algebra generated by the Yetter-Drinfeld module
$V(g,\rho)$ (see for example \cite{MR1714540}). 

\subsection{Representations of the Dihedral group}\label{sub:repdihe}
We present the Dihedral group $\mathbb{D}_n$ of order $2n$ by generators $r$ and
$s$ and relations 
$$r^{n}=s^{2}=1,\quad srs=r^{-1}.$$
When $n$ is even, the irreducible representations are given in Table
\ref{tab:diedral} (see \cite[Section 5.3]{MR0450380}).  We will not need the
Dihedral groups $\mathbb{D}_n$ with $n$ odd.

\begin{table}
\caption{Character table of $\mathbb{D}_n$ ($n$ even)}
\label{tab:diedral}
\begin{tabular}{|c|c|c|c|}
\hline
						   & Degree & $r^{k}$                  & $sr^{k}$\tabularnewline \hline
$\chi_{1}$                 & $1$    & $1$                      & $1$\tabularnewline \hline
$\chi_{2}$                 & $1$    & $1$                      & $-1$\tabularnewline \hline
$\chi_{3}$                 & $1$    & $(-1)^{k}$               & $(-1)^{k}$\tabularnewline \hline
$\chi_{4}$                 & $1$    & $(-1)^{k}$               & $(-1)^{k+1}$\tabularnewline \hline
$\mu_{h}$ ($0<h<\frac n2$) & $2$    & $2\cos\frac{2\pi hk}{n}$ & $0$\tabularnewline \hline
\end{tabular}
\end{table}

\subsection{Main tools}
Note that, since $\rho$ is irreducible and $g\in Z(C_{G}(g))$ (the center of
$C_G(g)$), then $\rho(g)$ is a scalar (by Schur lemma).  The following lemmas
are contained in \cite{ARXIV:0702559, ARXIV:0608701, ARXIV:0511020,
ARXIV:0703498, ARXIV:0605795}.

\begin{lem}
\label{lem:inversos}Assume that $\dim\mathfrak{B}(\mathcal{C},\rho)<\infty.$
If $g^{-1}\in\mathcal{C}$ then $\rho(g)=-1$.
\end{lem}
\begin{proof}
See \cite{ARXIV:0511020}.
\end{proof}
As in the cited papers, a conjugacy class is called \emph{real} if it contains
the inverses of its elements.

\begin{lem}
\label{lem:potencias}Assume that $\dim\mathfrak{B}(\mathcal{C},\rho)<\infty.$
If there exist $n>1$ such that $g^{n}\in\mathcal{C}$ then $\rho(g)=-1$
or $\rho(g)\in\mathcal{R}_{3}$. Moreover, if $g^{n^{2}}\ne g$ then
$\rho(g)=-1$.
\end{lem}
\begin{proof}
For the proof see for example \cite{ARXIV:0703498}. 
\end{proof}
As in the cited papers, a conjugacy class is called \emph{quasireal} if it contains
proper powers of its elements.

The next lemma is useful to treat, for instance, some conjugacy classes of
involutions. A particular case of this Lemma appears also in
\cite{fantino-2007}. 

\begin{lem}\label{lem:triangulitos}
	Let $G$ be a group.
	\begin{enumerate}
		\item\label{lem:trip1}
			Assume that $g_0,g_1,g_2\in G$ are conjugate and commute to each
			other, that $x_{1}x_{2}$ and $x_{2}x_{1}$ belong to $C_{G}(g_{0})$,
			(where $x_i$ are such that $g_{i}=x_{i}g_{0}x_{i}^{-1}$ for
			$i=1,2$), and that $g_1g_2=g_0^m$ for an odd integer $m$.
			Then $g_0$ is \bu.
		\item\label{lem:trip2}
			The conjugacy class of involutions in the alternating group
			$\mathbb{A}_4$ is \bu.
		\item\label{lem:trip3}
			If $g_0,x\in G$ are such that $g_0$ has order $2$ and both $x$ and
			$g_0x$ have order $3$, then $g_0$ is \bu.
	\end{enumerate}
\end{lem}
\begin{proof}
	Part \ref{lem:trip2} is a consequence of Part \ref{lem:trip1} by taking
	$g_0=(1\,2)(3\,4)$, $g_1=(1\,3)(2\,4)$, $g_2=(1\,4)(2\,3)$,
	$x_1=(2\,4\,3)$ and $x_2=x_1^{-1}$ (we have $g_1g_2=g_0$ in this way).
	Also, Part \ref{lem:trip3} is a consequence of Part \ref{lem:trip2} and
	Remark \ref{rem:tbym} since, as proved in \cite[Theorem 265]{MR0104735}
	(or use \textsf{GAP} to check it), $\mathbb{A}_4$ can be presented by generators
	$g_0$ and $x$ and relations
	\[
		g_0^2=x^3=(g_0x)^3=1.
	\]
	Therefore, we need to prove Part \ref{lem:trip1}.
	First notice that, since $g_{i}g_{j}=g_{j}g_{i}$, there exists $w\in V\setminus\{0\}$
	and $\lambda_i\in\mathbb{C}$ such that $\rho(g_{i})(w)=\lambda_{i}w$ for
	$i=0,1,2$. For any $0\leq i,j\leq2$, let
	$\gamma_{ij}=x_{j}^{-1}g_{i}x_{j}\in C_{G}(g_{0})$.  It is easy to see that
	\[
	\gamma=(\gamma_{ij})=\left(\begin{array}{ccc}
		g_{0} & g_{2} & g_{1}\\
		g_{1} & g_{0} & g_{1}^{m}g_{0}^{-1}\\
		g_{2} & g_{2}^{m}g_{1}^{-1} & g_{0}\end{array}\right).
	\]
	Then, $W=\text{span}\{x_{1}\otimes w,x_{2}\otimes w,x_{3}\otimes w\}$
	is a braided vector subspace of $M(\mathcal{C},\rho)$ of abelian type
	with Dynkin diagram given by
	\[
		\Dtriangle{$\lambda_0$}{$\lambda_0^m$}{$\lambda_0^{m^2-2}$}
	\]
	For $\mathfrak{B}(\mathcal{C},\rho)$ to be finite dimensional, we should have
	$\lambda_{0}=-1$ (see \cite[Table 3]{ARXIV:0509145}) and $m$ should be an even
	number, which contradicts our assumption.
\end{proof}

\section{Two examples: $M_{20}$ and $M_{21}$} In \cite{fantino-2007} the five
simple Mathieu groups are studied.  In this section we will study finite
dimensional Nichols algebras over the Mathieu groups $M_{20}$ and $M_{21}$
($=\mathbf{PSL}(3,4)$).  We begin with the non-simple Mathieu group $M_{20}$.
For a definition and elementary properties of this group, see \cite[Chapter
XII]{MR0224703}.  Since we do computations with $\textsf{GAP}$, we use the
product in the symmetric groups as they do: the product of two permutations
$\sigma$ and $\tau$ means the composition of the permutation $\sigma$ followed
by $\tau$. 

We know (see for example \cite{WebATLAS}) that
$M_{20}=\langle\alpha,\beta\rangle$ as a subgroup of $\mathbb{S}_{20}$, where
\begin{gather*}
\alpha=(1,2,4,3)(5,11,7,12)(6,13)(8,14)(9,15,10,16)(1,19,20,18),\\
\beta=(2,5,6)(3,7,8)(4,9,10)(11,17,12)(13,16,18)(14,15,19).
\end{gather*}

This is a group of order $960$. The conjugacy classes are 1A, 2A, 2B, 3A, 4A,
4B, 4C, 5A, 5B (we are using the \textsf{ATLAS} notation for conjugacy classes,
the name of a class begins with the order of its elements).

\begin{prop}\label{pro:m20}
	The Mathieu group $M_{20}$ is \bu.
\end{prop}

\begin{proof}
	By \cite[Remark 1.1]{ARXIV:0511020} we know that the class 1A gives 
	infinite dimensional Nichols algebras for every representation. 
	Using \textsf{GAP} it is easy to check that all conjugacy classes of $M_{20}$ 
	are real. Therefore, 
	by Lemma \ref{lem:inversos}, the conjugacy classes with representatives 
	of odd order (i.e. 3A, 5A, 5B) 
	will also give infinite dimensional Nichols algebras.

	All the remaining classes are dealt with Lemma~\ref{lem:triangulitos}. We
	just list the elements $g_1$, $g_2$ and $x_1$ for each of them. In all the
	cases, we put $g_0=g_1g_2$ and $x_2=x_1^{-1}$.
	We have used \textsf{GAP} to find those elements, with the help of some
	scripts which may be downloaded from the third author's web page:
	\verb|http://mate.dm.uba.ar/~lvendram/|.

	For class 2A, we take
	\begin{gather*}
		g_1=(2,3)(5,19)(6,13)(7,18)(8,14)(11,17)(12,20)(15,16), \\
		g_2=(2,15)(3,16)(5,18)(6,8)(7,19)(11,12)(13,14)(17,20),\\
		x_1=(1,4,10)(3,16,15)(5,11,14)(6,18,17)(7,12,8)(13,19,20).
	\end{gather*}
	For class 2B,
	\begin{gather*}
		g_1=(1,15)(2,10)(3,4)(5,13)(6,7)(8,18)(9,16)(14,19), \\
		g_2=(1,13)(2,7)(3,19)(4,14)(5,15)(6,10)(8,9)(16,18), \\
		x_1=(2,19,8)(3,18,6)(4,9,10)(5,13,15)(7,14,16)(12,17,20).
	\end{gather*}
	For class 4A,
	\begin{gather*}
		g_1=(1,2,10,15)(3,4,16,9)(5,14,7,8)(6,19,13,18)(11,12)(17,20), \\
		g_2=(1,5,4,19)(2,14,16,13)(3,6,15,8)(7,9,18,10)(11,17)(12,20), \\
		x_1=(2,6,5)(3,8,7)(4,10,9)(11,12,17)(13,18,16)(14,19,15).
	\end{gather*}
	For class 4B,
	\begin{gather*}
		g_1=(1,3,4,15)(2,10,16,9)(5,8,19,6)(7,14,18,13)(11,20)(12,17), \\
		g_2=(1,7,9,19)(2,6,3,14)(4,18,10,5)(8,15,13,16)(11,12)(17,20), \\
		x_1=(2,6,5)(3,8,7)(4,10,9)(11,12,17)(13,18,16)(14,19,15).
	\end{gather*}
	For class 4C,
	\begin{gather*}
		g_1=(1,16,9,15)(2,10,3,4)(5,6,18,14)(7,13,19,8)(11,17)(12,20), \\
		g_2=(1,18,10,19)(2,13,15,6)(3,8,16,14)(4,7,9,5)(11,20)(12,17), \\
		x_1=(2,6,5)(3,8,7)(4,10,9)(11,12,17)(13,18,16)(14,19,15).
	\end{gather*}
\end{proof}

We do the same now with $M_{21}$:
\begin{prop}
	The Mathieu group $M_{21}=\mathbf{PSL}(3,4)$ is \bu.
\end{prop}

\begin{proof}
	As in all groups, the trivial class 1A is \bu. 
	Classes 3A, 3B, 5A and 5B are real, as can be checked with \textsf{GAP}
	using the function \texttt{RealClasses}. Then these classes are \bu\ by
	Lemma~\ref{lem:inversos}. Classes 7A and 7B are quasireal, as can be checked
	with \textsf{GAP} with the function \texttt{PowerMaps}. Then these classes
	are \bu\ by Lemma~\ref{lem:potencias}.

	We are left with classes 2A, 4A, 4B, 4C. By \textsf{ATLAS}, we know that the
	semidirect product $2^4:\mathbb{A}_5=M_{20}$ is a maximal subgroup of
	$\mathbf{PSL}(3,4)$ (see \cite[pp. 23]{MR827219}). Now the result follows
	from \ref{pro:m20}, by looking at how classes of even order in $M_{20}$ sit
	into $M_{21}$, by using the function
	\texttt{PossibleClassFusions}\corr{}{:
	since $M_{20}$ is a subgroup of $M_{21}$ there exists a monomorphism
	$f:M_{20}\hookrightarrow M_{21}$ and therefore there exist a fusion of conjugacy
	classes. Moreover, with \texttt{GAP} we calculate all possible fusion
	of conjugacy classes and by inspection of all of these fusions it
	is easy to arrive to the desired result (see the file \texttt{fgv.log} 
	for the results).}
	
\end{proof}

\section{Nichols algebras over $\mathbf{PSL}(2,q)$} \label{se:psl2q}
In this section we deal with the groups
$\mathbf{PSL}(2,q)=\mathbf{SL}(2,\FF_q)/\{\pm I\}$, which have order
$\frac{(q-1)q(q+1)}{2}$.  Recall that if $q\ne2,3$ then $\mathbf{PSL}(2,q)$ is
simple (see eg. \cite[Theorem 2.6.8]{MR1369573}).
\corr{Since we know that
$\mathbf{PSL}(2,3)=\mathbb{A}_{4}$ and $\mathbf{PSL}(2,5)=\mathbb{A}_{5}$, we
may suppose that $q>5$.}{}

\begin{table}[ht]
\caption{Conjugacy classes of $\mathbf{\mathbf{PSL}}(2,\mathbb{F}_q)$ for $q\equiv1(4)$}
\label{tab:psl_q1}
\begin{tabular}{|c|c|c|c|}
\hline 
& Representative & Number & Size\tabularnewline \hline 
$\mathcal{C}_{1}$ & $c_{1}=\left(\begin{array}{cc} 1\\ & 1\end{array}\right)$
	& $1$ & $1$\tabularnewline \hline 
$\mathcal{C}_{2}$ & $c_{2}=\left(\begin{array}{cc} 1 & 1\\ & 1\end{array}\right)$
	& $1$ & $\frac{q^{2}-1}{2}$\tabularnewline \hline 
$\mathcal{C}_{3}$ & $c_{3}=\left(\begin{array}{cc} 1 & \Delta\\ & 1\end{array}\right)$
	& $1$ & $\frac{q^{2}-1}{2}$\tabularnewline \hline 
$\mathcal{C}_{4}$ & $c_{4}(z)=\left(\begin{array}{rr} x & \Delta y \\ y & x
\end{array}\right)$ ($N(z)=x^2-\Delta y^2=1$, $y\neq 0$)
	& $\frac{q-1}{4}$ & $q(q-1)$\tabularnewline \hline 
$\mathcal{C}_{5}$ & $c_{5}=\left(\begin{array}{cc} i\\ & -i\end{array}\right)$
	& $1$ & $\frac{q(q+1)}{2}$\tabularnewline \hline 
$\mathcal{C}_{6}$ & $c_{6}(x)=\left(\begin{array}{cc} x\\ & x^{-1}\end{array}\right)$ ($x\notin\{\pm 1,\pm i\}$)
	& $\frac{q-5}{4}$ & $q(q+1)$\tabularnewline \hline
\end{tabular}
\end{table}

\begin{table}
\caption{Conjugacy classes of $\mathbf{PSL}(2,q)$ for $q\equiv3(4)$}
\label{tab:psl_q3}
\begin{tabular}{|c|c|c|c|}
\hline 
& Representative & Number & Size\tabularnewline \hline 
$\mathcal{C}_{1}$ & $c_{1}=\left(\begin{array}{cc} 1\\ & 1\end{array}\right)$ & $1$ & $1$\tabularnewline \hline 
$\mathcal{C}_{2}$ & $c_{2}=\left(\begin{array}{cc} 1 & 1\\ & 1\end{array}\right)$ & $1$ & $\frac{q^{2}-1}{2}$\tabularnewline \hline 
$\mathcal{C}_{3}$ & $c_{3}=\left(\begin{array}{cc} 1 & \Delta \\ & 1\end{array}\right)$ & $1$ & $\frac{q^{2}-1}{2}$\tabularnewline \hline 
$\mathcal{C}_{4}$ & $c_{4}(z)=\left(\begin{array}{rr} x & \Delta y \\ y & x \end{array}\right)$ ($N(z)=x^2-\Delta y^2=1$, $xy\neq 0$) & $\frac{q-3}{4}$ & $q(q-1)$\tabularnewline \hline 
$\mathcal{C}_{5}$ & $c_{5}=\left(\begin{array}{cc} & -1\\ 1\end{array}\right)$ & $1$ & $\frac{q(q-1)}{2}$\tabularnewline \hline 
$\mathcal{C}_{6}$ & $c_{6}(x)=\left(\begin{array}{cc} x\\ & x^{-1}\end{array}\right)$ ($x\notin\{\pm 1\}$) & $\frac{q-3}{4}$ & $q(q+1)$\tabularnewline \hline
\end{tabular}
\end{table}

\subsection{The case $q\equiv1(4)$}
There are $\frac{q+5}{2}$ conjugacy classes, which are listed in Table~\ref{tab:psl_q1} (see \cite{jeffreyadams}).

\begin{prop}\label{pro:psl_c12345_q1}
	Classes $\mathcal{C}_i$ are \bu\ for $i=1,2,3,4,5$.
\end{prop}
\begin{proof}
	The trivial class is \bu\ for any group.
	For class $\mathcal{C}_2$ (resp. $\mathcal{C}_3$) we apply
	Lemma~\ref{lem:inversos}, since $c_2$ and $c_2^{-1}$ (resp. $c_3$ and
	$c_3^{-1}$) are conjugate and have odd order $p$.

	For $i=4$, we notice that
	$c_4(z)^{-1}=c_4(z^{-1})
		= \left( \begin{smallmatrix} x&-\Delta y\\-y&x \end{smallmatrix} \right)$,
	which is conjugate to $c_4(z)$. Indeed,
	$c_5c_4(z)c_5^{-1}=c_4(z^{-1})$. Then the result follows from
	Lemma~\ref{lem:inversos}, since the order of $c_4(z)$ is a divisor of
	$\frac{q+1}{2}$, which is odd.

	For $i=5$, we notice that $\mathbb{A}_4$ is a subgroup of
	$\mathbf{PSL}(2,q)$ (see \cite[Ch. XII]{MR0104735} or \cite[Satz II.8.18
	(b)]{MR0224703}). Since the only involutions in $\mathbf{PSL}(2,q)$ lie in
	class $\mathcal{C}_5$, Lemma~\ref{lem:triangulitos} part~\ref{lem:trip2} and
	Remark~\ref{rem:tbym} say that $\mathcal{C}_5$ is \bu.
\end{proof}

The centralizer of $c_6(x)$ is isomorphic to $\FF_q^*/\{\pm 1\}\simeq$
\corr{$\mathbb{Z}_{\frac{q-1}{2}}$}{$\mathbb{Z}/{\frac{q-1}{2}}$}.
\begin{prop}
	\label{pro:psl_c6_q1}
	If $\dim\mathfrak{B}(C_{6},\rho)<\infty$ then
	$c_{6}(x)$ has even order and $\rho(c_6(x))=-1$.
\end{prop}
\begin{proof}
	Again, $\mathcal{C}_6$ is a real class, since conjugating $c_6(x)$ by
	$\left( \begin{smallmatrix} &-1\\1& \end{smallmatrix} \right)$ gives
	$c_6(x^{-1})=c_6(x)^{-1}$. The result then follows from Lemma~\ref{lem:inversos}.
\end{proof}

\subsection{The case $q\equiv3(4)$}
There are \corr{$\frac{q-5}{2}$}{$\frac{q+5}{2}$} conjugacy classes, which are
listed in Table~\ref{tab:psl_q3} (see \cite{jeffreyadams}).

\begin{prop}\label{pro:psl_c12356_q3}
	Classes $\mathcal{C}_i$ are \bu\ for $i=1,2,3,5,6$.
\end{prop}
\begin{proof}
	It is entirely analogous to that of Proposition~\ref{pro:psl_c12345_q1}.
	The only difference is that now classes $\mathcal{C}_6$ have elements of odd order,
	instead of classes $\mathcal{C}_4$.
\end{proof}

The centralizer of $c_4(x)$ is isomorphic to \corr{$\mathbb{Z}_\frac{q+1}{2}$}{$\mathbb{Z}/\frac{q+1}{2}$}.
Again, by using Lemma~\ref{lem:inversos} we have
\begin{prop}\label{pro:psl_c6_q3}
	If $\dim\mathfrak{B}(\mathcal{C}_4,\rho)<\infty$, then $c_4(z)$ has even order and $\rho(c_4(z))=-1$.
\end{prop}

\section{Nichols algebras over $\mathbf{PGL}(2,q)$}\label{se:pgl2q}
We deal in this section with $\mathbf{PGL}(2,q)=\mathbf{GL}(2,q)/\{tI:t\in\mathbb{F}_q^\times\}$,
which has order $(q-1)q(q+1)$.
Since $\mathbf{PGL}(2,3)\simeq\mathbb{S}_{4}$ and
$\mathbf{PGL}(2,5)\simeq\mathbb{S}_5$ (which are treated in \cite{ARXIV:0511020})
we may assume that $q>5$. There are $q+2$
conjugacy classes, which are listed in Table~\ref{tab:pgl} (see \cite{jeffreyadams}).

\begin{table}[h]
	\caption{Conjugacy classes of $\mathbf{PGL}(2,q)$}
	\label{tab:pgl}
	\begin{tabular}{|c|c|c|c|}
		\hline 
		& Representative & Number & Size\tabularnewline
		\hline 
		$\mathcal{C}_{1}$ & $c_{1}=\left(\begin{array}{cc} 1 \\
			& 1\end{array}\right)$ & $1$ & $1$\tabularnewline
			\hline 
		$\mathcal{C}_{2}$ & $c_{2}=\left(\begin{array}{cc} 1 & 1\\
			& 1\end{array}\right)$ & $1$ & $q^{2}-1$\tabularnewline
			\hline 
		$\mathcal{C}_{3}$
			& $c_{3}(x)=\left(\begin{array}{cc} x\\ & 1\end{array}\right)$ $(x\ne\pm1)$
			& $\frac{q-3}{2}$ & $q(q+1)$\tabularnewline
			\hline 
		$\mathcal{C}_{4}$
			& $c_{4}=\left(\begin{array}{cc} -1\\ & 1\end{array}\right)$
			& $1$ & $\frac{q(q+1)}{2}$\tabularnewline
			\hline 
		$\mathcal{C}_{5}$
			& $c_{5}(z)=\left(\begin{array}{rr} x & \Delta y \\ y & x \end{array}\right)$ ($xy\neq 0$)
			& $\frac{q-1}{2}$ & $q(q-1)$\tabularnewline \hline 
		$\mathcal{C}_{6}$
			& $c_{6}=\left(\begin{array}{cc} & \Delta \\ 1\end{array}\right)$
			& $1$ & $\frac{q(q-1)}{2}$\tabularnewline
			\hline
	\end{tabular}
\end{table}

By considering the inclusion $\mathbf{PSL}(2,q)\hookrightarrow \mathbf{PGL}(2,q)$,
we get from the previous section that the following classes are \bu:
\begin{align*}
	&\mathcal{C}_1,\,\mathcal{C}_2 \\
	&\mathcal{C}_3 \quad\text{if } q\equiv 3\pmod 4 \text{ and } x\in\FF_q^2 \\
	&\mathcal{C}_4 \quad\text{if } q\equiv 1\pmod 4 \text{ (because it is class
			$\mathcal{C}_5$ in $\mathbf{PSL}(2,q)$)} \\
	&\mathcal{C}_5 \quad\text{if } q\equiv 1\pmod 4 \text{ and } N(z)=x^2-\Delta y^2\in\FF_q^2 \\
	&\mathcal{C}_6 \quad\text{if } q\equiv 3\pmod 4 \text{ (because it is class
			$\mathcal{C}_5$ in $\mathbf{PSL}(2,q)$)}
\end{align*}

We consider now the other cases. They are not \bu, but one can restrict the
representations anyway.

The centralizer of $c_3(x)$ is given by the classes in $\mathbf{PGL}(2,q)$ of
matrices $\left( \begin{smallmatrix}\sigma&\\&1\end{smallmatrix} \right)$
and therefore it is isomorphic to $\mathbb{F}_q^\times$.
\begin{prop}\label{pro:pgl_c3}
	Let $q\equiv 1\pmod 4$.
	If $\dim\mathfrak{B}(\mathcal{C}_{3},\rho)<\infty$, then $x$ has even order
	and $\rho(c_3(x))=-1$.
\end{prop}
\begin{proof}
	Notice that $\mathcal{C}_3$ is a real class, and use Lemma~\ref{lem:inversos}.
\end{proof}

The centralizer of $c_4$ is given by the classes in $\mathbf{PGL}(2,q)$ of
matrices
$\left( \begin{smallmatrix} 1&\\&\sigma \end{smallmatrix} \right)$ and
$\left( \begin{smallmatrix} &1\\\sigma& \end{smallmatrix} \right)$ for
$\sigma\in\mathbb{F}_q^\times$. It is easy to see that this group is isomorphic
to $\mathbb{D}_{q-1}$. In the presentation by $r,s$ given at
\S\ref{sub:repdihe}, $c_4$ corresponds to $r^{\frac{q-1}{2}}$.

\begin{prop}\label{pro:pgl_c4}
	Let $q\equiv3\pmod 4$ and let $\rho$ be a representation of
	$C_{\mathbf{PGL}(2,q)}(c_4)$ with character $\chi$.
	If $\dim\mathfrak{B}(\mathcal{C}_4)<\infty$, then
	$\chi\in\{\chi_{3},\chi_{4},\mu_{h}\text{ ($h$ odd)}\}$ (see
	Table~\ref{tab:diedral}).
\end{prop}
\begin{proof}
	Since $c_4$ is an involution, we must have $\rho(c_4)=-1$. Then, the result
	follows by inspection of the Table.
\end{proof}

The centralizer of $c_5(z)$ is isomorphic to $E^*/\FF_q^*\simeq\mathbb{Z}/{q+1}$
(see \cite{jeffreyadams}).
\begin{prop}\label{pro:pgl_c5}
	Let $q\equiv 3\pmod 4$. If $\dim\mathfrak{B}(\mathcal{C}_{5},\rho)<\infty$,
	then $c_{5}(z)$ has even order and $\rho(g)=-1$.
\end{prop}
\begin{proof}
	In a similar way as the proof of Prop.~\ref{pro:psl_c12345_q1}, we see that
	classes $\mathcal{C}_5$ are real by conjugating $c_5(z)$ with $c_4$. Now, we
	apply Lemma~\ref{lem:inversos}.
\end{proof}

The centralizer of $c_6$ is given by the classes in $\mathbf{PGL}(2,q)$ of matrices
$\left(\begin{smallmatrix} x & \pm\Delta y \\ y & x \end{smallmatrix}\right)$.
It is easy to see that this group is isomorphic to $\mathbb{D}_{q+1}$. In the
presentation by $r,s$ given at \S\ref{sub:repdihe}, $c_6$ corresponds to
$r^{\frac{q+1}{2}}$.

\begin{prop}\label{pro:pgl_c6}
	Let $q\equiv1\pmod4$ and let $\chi$ be the character of the representation
	$\rho$. If $\dim\mathfrak{B}(\mathcal{C}_6)<\infty$, then
	$\chi\in\{\chi_{3},\chi_{4},\mu_{h}\text{ ($h$ odd)}\}$ (see
	Table~\ref{tab:diedral}).
\end{prop}
\begin{proof}
	Analogous to that of Prop.~\ref{pro:pgl_c4}.
\end{proof}

\section{The class of order $4$ in $\mathbf{SL}(2,q)$}
In this section we use the tools in the present paper to improve a result in
\cite{ARXIV:0703498}. There, we considered in $\mathbf{SL}(2,q)$ the classes
$\mathcal{C}_7$ of
$c_7(x)=\left( \begin{smallmatrix} x&0 \\ 0&x^{-1} \end{smallmatrix} \right)$,
and $\mathcal{C}_8$ of
$c_8(z)=\left( \begin{smallmatrix} x&\Delta y \\ y&x \end{smallmatrix} \right)$,
where $z=x+\delta y\in E\setminus\FF_q$.  Propositions~3.4 (resp. 3.5) in
\cite{ARXIV:0703498} prove that in order for $\toba(\mathcal{C}_7)$
(resp. $\toba(\mathcal{C}_8)$) to be
finite dimensional, the order of $x$ (resp. $z$) must be even.  On the other hand, when
$q=3$ there is only one class $\mathcal{C}_8$, and it is proved in \cite{afgv}
that it is \bu.  We prove here that $\mathbf{SL}(2,3)$ is a subgroup of
$\mathbf{SL}(2,q)$. This implies that when $q\equiv1\pmod4$ and $x$ has order
$4$, the class of $c_7(x)\in\mathbf{SL}(2,q)$ is \bu, and when $q\equiv3\pmod4$
and $z$ has order $4$, the class of $c_8(z)\in\mathbf{SL}(2,q)$ is \bu.

One possible presentation of $\mathbf{SL}(2,3)$ is given by generators $x,y$ and
relations $x^{3}=y^{4}=(xy^{3})^{3}=1$, $xy^{2}=y^{2}x$. In fact, one can take
\[
	x = \begin{pmatrix} 1 & 1 \\ 0 & 1 \end{pmatrix}, \qquad
	y = \begin{pmatrix} 1 & 1 \\ 1 & 2 \end{pmatrix}.
\]
Then, if a group $G$ has elements $A,B,C$ of orders $3$,
$3$ and $4$ respectively such that $ABC=[B,C^2]=1$, then $G$ has a subgroup
isomorphic to $\mathbf{SL}(2,3)$. Indeed, take the map $x\mapsto B$, $y\mapsto
C^3$. It is easy to see that this defines a map $\mathbf{SL}(2,3)\to G$. It is
injective because otherwise the orders of $A$, $B$ or $C$ would be smaller than
stated, as can be seen by considering the normal subgroups of $\mathbf{SL}(2,3)$
(the non-trivial ones being isomorphic to $C_2$ and the quaternion group $Q$ of
order $8$). Also, any $C\in\mathbf{SL}(2,q)$ has order $4$ if and only if
$\tr C=0$. Indeed, it is easy to see that if $\alpha=\tr C$, then
$C^4=(\alpha^3-2\alpha)C+(1-\alpha^2)$, which implies the claim. Therefore,
there is only one class in $\mathbf{SL}(2,q)$ of order $4$. This means that the
embedding we shall find $\mathbf{SL}(2,3)\hookrightarrow\mathbf{SL}(2,q)$ sends
the class of
$\left( \begin{smallmatrix} 0&-1 \\ 1&0 \end{smallmatrix} \right)\in\mathbf{SL}(2,3)$
\bu\ to it.

We prove now the existence of the embedding
$\mathbf{SL}(2,3)\hookrightarrow\mathbf{SL}(2,q)$.  Since $\mathbf{SL}(2,p)$
embeds in $\mathbf{SL}(2,q)$, it is enough to prove that $\mathbf{SL}(2,3)$
embeds in $\mathbf{SL}(2,p)$. Let $x\in\FF_p^\times$ be of
order $p-1$ and let $\mathcal{G}_l$ be the class of $c_7(x^l)$. Let $z\in E$
generate the group $\{a+\delta b\in E\;|\;a^2-\Delta b^2=1\}$ and let
$\mathcal{H}_m$ be the class of $c_8(z^m)$.
The character table of $\mathbf{SL}(2,p)$ restricted to the classes $\mathcal{G}_l$
and $\mathcal{H}_m$ is given in Table~\ref{tab:carslgh}, where we
write $c[\frac ab]=e^{\frac{2\pi ia}b}+e^{-\frac{2\pi ia}b}$.
There, $1\le i\le\frac{p-3}{2}$ for $\zeta_i$ and $1\le i\le\frac{p-1}{2}$ for $\theta_i$.
\begin{table}[h]
	\caption[Longtable]{Character table on classes $\mathcal{G}_l$ and $\mathcal{H}_m$.}
	\label{tab:carslgh}
	\begin{tabular}{|c|c|c|c|c|c|c|c|c|} \hline
		& $1$						  & $\psi$			 & $\zeta_{i}$		  & $\xi_{1}$
		& $\xi_{2}$					& $\theta_{i}$	   & $\eta_{1}$
		& $\eta_{2}$ \\ \hline
		deg			 & $1$						  & $p$				& $p+1$
		& $\frac{1}{2}(p+1)$		   & $\frac{1}{2}(p+1)$ & $p-1$				& $\frac{1}{2}(p-1)$
		& $\frac{1}{2}(p-1)$ \\ \hline
		$\mathcal{G}_l$ & $1$						  & $1$				& $c[\frac{il}{p-1}]$
		& $(-1)^{l}$				   & $(-1)^{l}$		 & $0$				  & $0$
		& $0$ \\ \hline
		$\mathcal{H}_m$ & $1$						  & $-1$			   & $0$
		& $0$						  & $0$				& $-c[\frac{im}{p+1}]$ & $(-1)^{m+1}$
		& $(-1)^{m+1}$ \\ \hline
	\end{tabular}
\end{table}

When $p\equiv1\pmod3$, classes $\mathcal{G}_l$ have order $3$ for $l=\frac{p-1}{3}$ and $l=\frac{2(p-1)}{3}$.
On the other hand, when $p\equiv2\pmod3$, classes $\mathcal{H}_m$ have order $3$
for $m=\frac{p+1}{3}$ and $m=\frac{2(p+1)}{3}$. We use then classes
$\mathcal{G}_l$ and $\mathcal{H}_m$ for $A$, $B$ or $C$ depending on the class
of $p\pmod{12}$.

We apply then the well-known formula (see e.g. \cite[Theorem 2.12]{MR0231903})
\[
S(\mathcal{C}_{i},\mathcal{C}_{j},\mathcal{C}_{k})
	=\frac{|\mathcal{C}_{i}||\mathcal{C}_{j}||\mathcal{C}_k|}{|G|}
		\sum_{\chi}\frac{\chi(\mathcal{C}_{i})\chi(\mathcal{C}_{j})\chi(\mathcal{C}_{k})}{\chi(1)}
\]
which counts solutions of the equation $abc=1$ with $a$, $b$, $c$ respectively
in classes $\mathcal{C}_i$, $\mathcal{C}_j$ and $\mathcal{C}_k$. Therefore, it
is enough to see that for any $p$ we have
$S(\mathcal{C}_{i},\mathcal{C}_{j},\mathcal{C}_{k})>0$, where
classes $\mathcal{C}_i$, $\mathcal{C}_j$ and $\mathcal{C}_k$ have orders $3$,
$3$ and $4$. We have then the following possibilities:

\begin{itemize}
	\item If $p\equiv5\pmod{12}$, the element of order $4$ belongs
		to the class $\mathcal{G}$ with $l=\frac{p-1}{4}$ and the element
		of order $3$ belongs to the class $\mathcal{H}$ with $m=\frac{p+1}{3}$.
		Then,
		\[
		S(\mathcal{H},\mathcal{H},\mathcal{G})
		=\frac{p^{3}(p-1)^{2}(p+1)}{p(p-1)(p+1)}\left(1+\frac{1}{p}\right)=(p-1)p(p+1)>0.
		\]

	\item If $p\equiv7\pmod{12}$, the element of order $4$ belongs
		to the class $\mathcal{H}$ with $m=\frac{p+1}{4}$ and the element
		of order $3$ belongs to the class $\mathcal{G}$ with $l=\frac{p-1}{3}$.
		Then,
		\[
		S(\mathcal{G},\mathcal{G},\mathcal{H})
		=\frac{p^{3}(p+1)^{2}(p-1)}{p(p-1)(p+1)}\left(1-\frac{1}{p}\right)=(p-1)p(p+1)>0.
		\]

	\item If $p\equiv1\pmod{12}$
		\begin{align*}
			& S(\mathcal{G}_{\frac{p-1}3},\mathcal{G}_{\frac{p-1}3},\mathcal{G}_{\frac{p-1}4}) \\
			&\hspace{.6cm} = \frac{p^2(p+1)^2}{p-1}\left(1+\frac{1}{p}
				+\sum_{j=1}^{\frac{p-3}{2}} \frac{1}{p+1} c[\frac{j}{3}]^2 c[\frac{j}{4}]
				+ 4 \frac{(-1)^{\frac{p-1}{4}}}{p+1} \right) \\
			&\hspace{.6cm} \ge \frac{p^2(p+1)^2}{p-1}\left(1+\frac{1}{p}+ \sum_{j=1}^{\frac{p-3}{2}} \frac{8}{p+1}
				\cos(\frac{2\pi j}{3})^2 \cos(\frac{2\pi j}{4}) -\frac{4}{p+1}\right)\\
			&\hspace{.6cm} = \frac{p^2(p+1)^2}{p-1}\left(1+\frac{1}{p} + \frac{8}{p+1} \sum_{l=1}^{[\frac{p-3}{4}]} 
				\cos(\frac{4\pi l}{3})^2 \cos({\pi l}) -\frac{4}{p+1}\right) \\
			&\hspace{.6cm} = \frac{p^2(p+1)^2}{p-1}\left(1+\frac{1}{p} + \frac{8}{p+1}
				\Big(\sum_{\substack{l=1 \\ l\equiv 0 (3)}}^{[\frac{p-3}{4}]}
				\cos({\pi l}) + \sum_{\substack{l=1 \\ l\not\equiv 0 (3)}}^{[\frac{p-3}{4}]} \frac{1}{4} \cos({\pi l}) \Big)
				-\frac{4}{p+1}\right) \\
			&\hspace{.6cm} = \frac{p^2(p+1)^2}{p-1}\left(1+\frac{1}{p} +
				\frac{2}{p+1} \Big(\sum_{\substack{l=1 \\ l\equiv 0 (3)}}^{[\frac{p-3}{4}]}
				3 \cos({\pi l}) + \sum_{l=1} ^{[\frac{p-3}{4}]}  \cos({\pi l}) \Big) -\frac{4}{p+1}\right) \\
			&\hspace{.6cm} \ge \frac{p^2(p+1)^2}{p-1} \Big(1+\frac{1}{p} - \frac{12}{p+1}\Big) > 0
		\end{align*}

	\item If $p\equiv11\pmod{12}$
		\begin{align*}
			& S(\mathcal{H}_{\frac{p+1}3},\mathcal{H}_{\frac{p+1}3},\mathcal{H}_{\frac{p+1}4}) \\
			&\hspace{.6cm} = \frac{p^2(p-1)^2}{p+1}\left(1-\frac{1}{p}
				-\sum_{j=1}^{\frac{p-1}{2}} \frac{1}{p-1} c[\frac{j}{3}]^2 c[\frac{j}{4}]
				+ 4 \frac{(-1)^{\frac{p+5}{4}}}{p-1} \right) \\
			&\hspace{.6cm} \ge \frac{p^2(p-1)^2}{p+1}\left(1-\frac{1}{p}- \sum_{j=1}^{\frac{p-1}{2}} \frac{8}{p-1}
				\cos(\frac{2\pi j}{3})^2 \cos(\frac{2\pi j}{4}) -\frac{4}{p-1}\right)\\
			&\hspace{.6cm} = \frac{p^2(p-1)^2}{p+1}\left(1-\frac{1}{p} - \frac{8}{p-1} \sum_{l=1}^{[\frac{p-1}{4}]} 
				\cos(\frac{4\pi l}{3})^2 \cos({\pi l}) -\frac{4}{p-1}\right) \\
			&\hspace{.6cm} = \frac{p^2(p-1)^2}{p+1}\left(
				1-\frac{1}{p} - \frac{8}{p-1} \Big(\sum_{\substack{l=1 \\ l\equiv 0 (3)}}^{[\frac{p-1}{4}]}
				\cos({\pi l}) + \sum_{\substack{l=1 \\ l\not\equiv 0 (3)}}^{[\frac{p-1}{4}]} \frac{1}{4} \cos({\pi l})
				\Big) -\frac{4}{p-1}\right) \\
			&\hspace{.6cm} = \frac{p^2(p-1)^2}{p+1}\left(
				1-\frac{1}{p} - \frac{2}{p-1} \Big(\sum_{\substack{l=1 \\ l\equiv 0 (3)}}^{[\frac{p-1}{4}]}
				3 \cos({\pi l}) + \sum_{l=1} ^{[\frac{p-1}{4}]} \cos({\pi l}) \Big) -\frac{4}{p-1}\right) \\
			&\hspace{.6cm} \ge \frac{p^2(p-1)^2}{p+1} \Big(1-\frac{1}{p} - \frac{4}{p-1}\Big) > 0
		\end{align*}
\end{itemize}

\begin{acknowledgement*}
	We thank Nicol\'as Andruskiewitsch and Fantino Fantino for fruitful
	discussions.
	We have used \textsf{GAP} \cite{GAP} to do some of the computations.
\end{acknowledgement*}

\end{document}